\documentclass[12pt]{article}
\usepackage{amsmath,amssymb,amsthm,amscd}
\usepackage[dvipsnames]{xcolor}
\usepackage[hypertexnames=false,linkcolor=black, colorlinks=true, citecolor=black, filecolor=black,urlcolor=black]{hyperref}
\usepackage{graphicx}
\usepackage{tikz}

\numberwithin{equation}{section}

\newtheorem{theorem}{Theorem}[section]
\newtheorem{prop}[theorem]{Proposition}
\newtheorem{lem}[theorem]{Lemma}
\newtheorem{Rk}[theorem]{Remark}

\newtheorem{definition}[theorem]{Definition}

\newtheorem{claim}[theorem]{Claim}
\newtheorem{conventions}[theorem]{Conventions}

\newcommand{\HS}{\operatorname{HS}}
\newcommand{\ord}{\operatorname{ord}}
\newcommand{\Rid}{\operatorname{Rid}}
\newcommand{\Sing}{\operatorname{Sing}}
\newcommand{\Spec}{\operatorname{Spec}}
\newcommand{\Frac}{\operatorname{Frac}}

\newcommand{\N}{\mathbb{N}}
\newcommand{\Z}{\mathbb{Z}}
\newcommand{\A}{\mathbb{A}}

\newcommand{\ini}{\operatorname{in}}
\newcommand{\gr}{\operatorname{gr}}
\newcommand{\Proj}{\operatorname{Proj}}
\newcommand{\car}{\mathrm{char}}
\newcommand{\Dir}{\mathrm{Dir}}

\newcommand{\fm}{\mathfrak{m}}
\newcommand{\fn}{\mathfrak{n}}

\newcommand{\Bl}{\operatorname{Bl}}

\newcommand{\cP}{\mathcal{P}}

\begin{document}

\title{Constancy of the Hilbert-Samuel function}

\author{Vincent Cossart \footnotemark \hspace{20pt} Olivier Piltant \footnotemark[1] \hspace{15pt} Bernd Schober  \footnotemark[2] }

\footnotetext[1]{Laboratoire de Math\'{e}matiques de Versailles LMV UMR 8100, Universit\'{e} de Paris-Saclay,
45, avenue des \'{E}tats-Unis, 78035 VERSAILLES Cedex, France\\
{\tt vincent.cossart@uvsq.fr, olivier.piltant@uvsq.fr}}
\footnotetext[2]{Institut f\"ur Mathematik
	Carl von Ossietzky, Universit\"at Oldenburg,
	26111 Oldenburg (Oldb),
	Germany	\\
        {\em Current address:} None. (Hamburg, Germany). \\
{\tt schober.math@gmail.com}
\\[5pt]
}

\maketitle

\begin{abstract}
	We prove a criterion for the constancy of the Hilbert-Samuel function for locally Noetherian schemes such that the local rings are excellent at every point.
	More precisely, we show that the Hilbert-Samuel function is locally constant on such a scheme if and only if the scheme is normally flat along its reduction and the reduction itself is regular. Regularity of the underlying reduced scheme is a significant new property.
\end{abstract}

\section{Introduction}

The Hilbert-Samuel function and the multiplicity function are fundamental locally defined invariants on Noetherian schemes. They have been playing an important role in desingularization for many years  \cite{Z,H,B,Gi1,V,BM,CoJS}, and others. Bennett studied upper  semicontinuity of the Hilbert-Samuel function on schemes and linked it with Hironaka's invariant $\nu^*$. He also  proved that it is non increasing under permissible blowing ups.
The latter are blowing ups at regular subschemes along which the singular scheme is normally flat.
For the definition of the Hilbert-Samuel function of $ \mathcal{X} $ and the notion of normal flatness, we refer to Definitions~\ref{Def:HS} and~\ref{Def:normalflat_perm}(1), respectively.

For a reduced scheme, the Hilbert-Samuel function is locally constant if and only if it is regular: this translates the question of resolution of singularities into  a problem of lowering the  Hilbert-Samuel function. We show here that,  this result can be  extended to \textit{non reduced} schemes, as follows.

\begin{theorem}\label{th:th}
Let $\mathcal X$ be a locally Noetherian  scheme such that $\mathcal{O}_{\mathcal{X},x}$ is excellent for every $x\in \mathcal{X}$.
The Hilbert-Samuel function is locally constant on $\mathcal{X}$ if and only if   $\mathcal{X}_{red}$ is everywhere regular and
$\mathcal{X}$ is normally flat on $\mathcal{X}_{red}$.
\end{theorem}

In the case where $\mathcal{X}$ is not reduced of characteristic~$0$ or of dimension~$\leq 2$, using Hironaka's results, there exists a projective morphism composition of permissible blowing ups (cf.~Definition~\ref{Def:normalflat_perm}(2))
$$ \mathcal{X}' \longrightarrow \mathcal{X} $$
such that the Hilbert-Samuel function is locally constant on $\mathcal{X}'$. As a consequence of our theorem, there exists a projective morphism
$  \mathcal{X}' \longrightarrow \mathcal{X} $
such that $\mathcal{X}'_{red}$ is everywhere regular and
$ \mathcal{X}' $ is normally flat along $\mathcal{X}_{red}'$. See  \cite[Corollary~6.19 and Remark~6.20]{CoJS} and \cite[Claim 2.11]{BEV}. At any rate, this sheds a new light on \cite[Theorem  I*, page 138]{H1} reformulated in  \cite[Theorem 11.14]{BM}.


Let us note that in \cite[Proposition~6.14]{V}, O.~Villamayor gets a similar and finer result for the multiplicity function instead of the Hilbert-Samuel function, in the case where  $\mathcal X$ is an equidimensional scheme of finite type over a perfect field $k$: the strata defined by the multiplicity on $\mathcal{X}_{red}$ and $\mathcal{X}$  coincide. This is achieved from a local algebraic and differential description of the maximal multiplicity locus of a variety. See also \cite{A}.
Nonetheless, the hypotheses of Theorem~\ref{th:th} are much more general.

\smallskip

Apart from classical results on Hilbert-Samuel functions, this article makes essential use of Hironaka's characteristic
polyhedron \cite{H} which is an essential ingredient in the proof of Theorem \ref{th:th}. These notions and relevant
properties  are recalled in the next section.

A substantial difficulty here consists in dealing with
nonperfect residue fields of positive characteristic. For instance, Theorem \ref{th:th} is not obvious even  for
cones in affine space over such a field. To overcome this, we recall part of Giraud's theory of
presentations \cite{Gi2} in Section 3 and derive as a byproduct a general result on the Hilbert-Samuel stratum for affine
cones, see Proposition~\ref{Pr: HSconeridge}. The notions of directrix and ridge of a cone play an important role
and we recall them beforehand in Section 2 (Definition \ref{defdirridge}). Finally the proof of
Theorem \ref{th:th} is given in the last section,
 where the main argument involves characteristic polyhedra.

\medskip

\noindent {\it Acknowledgment:} The authors thank A.~Bravo, S.~Encinas, M.~Spivakovsky and O.~Villamayor for fruitful discussions on the topic.

\bigskip

\begin{conventions}
	We follow the convention that zero is a natural number,
	$ \N = \mathbb{Z}_{\geq 0 }$.
	Further, we denote by $ \N^\N $ the set of functions $ \N \to \N $.
	We equip $ \N^\N $ with the product ordering:
	for $ \alpha, \beta \in \N^\N $, we write $ \alpha \leq \beta $ if and only if $ \alpha(n) \leq \beta(n) $ for every $ n \in \N $.
	\\
	We use multi-index notation. For example, we write $ X^A = X_1^{A_1} \cdots X_n^{A_n} $ for $ X = (X_1, \ldots, X_n) $ and $ A = (A_1, \ldots, A_n ) \in \N^n $.
	\\
	All schemes are assumed to be finite dimensional.
\end{conventions}

\section{Invariants of singularities}

Fix a locally Noetherian scheme $ \mathcal{X} $ and a point $ x \in \mathcal{X} $ (not necessarily closed).
We begin by recalling the definition of local invariants for the singularity of $ \mathcal{X} $ at $ x $.
For a reference providing more details on the different notions, we refer to \cite[Chapters~2 and~3]{CoJS}.

\begin{definition}[{\cite[after Lemma~2.22 and Definition~2.28]{CoJS}}] \hfill
		\label{Def:HS}
	\phantomsection

	\begin{enumerate}
	
		\item
		Let $ \mathcal{O} $ be a Noetherian local ring (which is not necessarily regular) with maximal ideal $ \fn $ and residue field $ \kappa $.
		The {\em Hilbert-Samuel function $ H_{\mathcal{O}}^{(0)} \colon \N \to \N $
		of $ \mathcal{O} $}
		is defined by
		\[
			H_{\mathcal{O}}^{(0)} (n) := \dim_\kappa  (\fn^n / \fn^{n+1}),
			\ \ \ \mbox{ for } n \in \N.
		\]
		Furthermore, one defines $ H_{\mathcal{O}}^{(t)} \colon \N \to \N $, for $ t \in \N $ and $ t > 0 $, via the recursion
		\[
			H_{\mathcal{O}}^{(t)} (n) := \sum_{i=0}^{n} H_{\mathcal{O}}^{(t-1)} (i),
			\ \ \ \mbox{ for } n \in \N.
		\]
		
		\item
		Let $ \mathcal{X} $ be a locally Noetherian catenary scheme and fix
		$ N \geq \dim (\mathcal{X}) $.
	 	The {\em modified Hilbert-Samuel function $ H_{\mathcal{X}} := H_{\mathcal{X}}^N \colon \mathcal{X} \to \N^\N $} is defined by
	 	\[
	 		H_{\mathcal{X}} (x) := H_{\mathcal{O}_{\mathcal{X},x}}^{(\phi_{\mathcal{X}}(x))} \in \N^\N ,
	 	\]
	 	where $ \phi_{\mathcal{X}} (x) := N - \min\{ \operatorname{codim}_{\mathcal{Y}} (x) \mid \mathcal{Y} \in \operatorname{Irr}(x) \} $
	 	and $ \operatorname{Irr}(x) $ denotes the set of irreducible components of $ \mathcal{X} $ containing $ x $.
	 	Often, we call $ H_{\mathcal{X}} $ just the Hilbert-Samuel function of $ \mathcal{X} $.

	 	Furthermore, we define, for $ \nu \in \N^\N $,
	 	\[
	 		\begin{array}{l}
	 		\mathcal{X}( \geq \nu) := \{ x \in \mathcal{X} \mid H_{\mathcal{X}} (x) \geq \nu \}
	 		\\
	 		\mathcal{X}( \nu) := \{ x \in \mathcal{X} \mid H_{\mathcal{X}} (x)  = \nu \} 	
	 		\end{array}
	 	\]
	 	and call $ \mathcal{X}(\nu) $ the {\em Hilbert-Samuel stratum of $ \mathcal{X} $ for $ \nu $}.
 \end{enumerate}
\end{definition}

Using the Cohen Structure Theorem for complete Noetherian local rings,
we may assume when necessary that at $ x \in \mathcal{X} $, there is an embedding $ (\mathcal{X} ,x)\subset (\mathcal{Z},x) $ in a regular $ \mathcal{Z} $.
Hence, we reduce the setting to a non-zero ideal $ \mathcal{I} \subset R $  in a (complete) regular local ring $ R $.
More precisely, $ R := \mathcal{O}_{\mathcal{Z},x} $ is the local ring of $ \mathcal{Z} $ at $ x $ and $ \mathcal{I}  $ is the ideal describing $ \mathcal{X} $ locally at $ x $.
We denote by $ \fm \subset R $ the maximal ideal of $ R $ and by $ k := R/\fm $ the residue field of $ R $.

The {\em graded ring of $ R $ at $ \fm $} is defined as
$ \gr_\fm (R) := \bigoplus_{s \geq 0} \fm^s/\fm^{s+1} $.
If we fix a regular system of parameters $ (z) = (z_1, \ldots, z_n) $ for $ R $,
then there is an isomorphism $ \gr_\fm (R) \cong k[X_1,\ldots,X_n] $,
where $ X_i := z_i \mod \fm^2 $ for $ i \in \{ 1, \ldots, n \} $ and $ n := \dim(R) $.

\begin{definition}
	Let $ (R, \fm, k ) $ be a regular local ring.
	\begin{enumerate}
		\item
		Let $ f \in R \setminus \{ 0 \} $ be a non-zero element in $ R $.
		The {\em initial form $ \ini_\fm(f) \in \gr_\fm (R) $ of $ f $ (with respect to $ \fm $)} is defined as the class of $ f $ in $ \fm^{d} / \fm^{d+1} $,
		where $ d := \ord_\fm (f) := \sup\{ a \in \N \mid f \in \fm^a \} $ is the {\em order of $ f $ at $ \fm $}.
		Moreover, we set $ \ini_\fm (0) := 0 $.
		
		\item
		Let $ \mathcal{I}  \subset R $ be a non-zero ideal.
		The {\em ideal of initial forms $ \ini_\fm (\mathcal{I} ) $ of $ \mathcal{I}  $ (with respect to $ \fm $)} is defined as the ideal
		in $ \gr_\fm ( R ) $ generated by the initial forms of all $ f \in \mathcal{I}  $,
		\[
			\ini_\fm (\mathcal{I} ) := ( \ini_\fm (f) \mid f \in \mathcal{I}  ) \subseteq \gr_\fm ( R )
		\]
	\end{enumerate}
\end{definition}

In our setting, $ x \in \mathcal{X} \subset \mathcal{Z} $ with $ \mathcal{O}_{\mathcal{X},x} \cong R/\mathcal{I}  $,
the initial ideal defines the {\em tangent cone $ C_{x} (\mathcal{X}) $ of $ \mathcal{X} $ at $ x $},
\[
C_{x} (\mathcal{X}) := \Spec (\gr_\fm (R) / \ini_\fm (\mathcal{I} ) ) \subseteq  \Spec (\gr_\fm (R) ) \cong  \mathbb{A}^n_k ,
\]
which is a first approximation of the singularity of $ \mathcal{X} $ at $ x $.

\begin{definition}
	\label{Def:std_basis_nu_ast}
	Let $ \mathcal{X} $ be a locally Noetherian scheme, which is embedded in a regular $ \mathcal{Z} $, and let $ x \in \mathcal{X} $ be any point.
	Let $ \mathcal{I} \subset R $ be the ideal defining $ X $ locally at $ x $, where $ (R := \mathcal{O}_{\mathcal{Z},x}, \fm, k ) $ is the respective regular local ring.
	\begin{enumerate}
		\item
		A system of elements $ (f) = (f_1, \ldots, f_m) \in R^m $  is called a {\em standard basis of $ \mathcal{I}  $ at $ \fm $} if
		\begin{enumerate}
			\item
			$ \ini_\fm (\mathcal{I} ) = ( \, \ini_\fm (f_1), \ldots, \ini_\fm (f_m) \, ) $,
			
			\item
			 $ \ini_\fm (f_i) \notin ( \, \ini_\fm (f_1), \ldots, \ini_\fm (f_{i-1}) \, ) $, for every $ i \in \{ 2, \ldots, m \} $,
			 and
			
			 \item
			 if we set $ \nu_i := \ord_\fm (f_i) $ for $ i \geq 1 $, we have  $ \nu_1 \leq \nu_2 \leq \cdots \leq \nu_m $.
		\end{enumerate}
		
		\item
		The {\em $ \nu^* $-invariant of $ \mathcal{X} \subset \mathcal{Z} $ at $ x $} is defined as
		\[
			\nu^*_x ( \mathcal{X}, \mathcal{Z} )
			:= \nu^* _\fm (\mathcal{I} , R )
			:= (\nu_1, \ldots, \nu_m, \nu_{m+1}, \ldots, ),
		\]
		where $ (\nu_1, \ldots, \nu_m ) $ is determined by a standard basis of $ \mathcal{I}  $ at $ \fm $ and $ \nu_j = \infty $ for all $ j \geq m+1 $.
	\end{enumerate}
\end{definition}

Even though the definition of the $ \nu^* $-invariant seems to depend on the choice of a standard basis, this is not the case.
For details, we refer to \cite[Chapter~2]{CoJS} or \cite{H2}.
Following Hironaka, we denote:
\begin{equation}\label{eq:nu*}
\nu^*_x ( \mathcal{X}, \mathcal{Z} ) =:(\nu^1 _x (\mathcal{I} , R ) , \ldots,\nu^m _x (\mathcal{I} , R ) ,\nu^{m+1} _x (\mathcal{I} , R ) ,\ldots, ).
\end{equation}

Hironaka's $ \nu^* $-invariant is an invariant measuring the complexity of the singularity of $ \mathcal{X} $ at $ x $,
which is closely related to the Hilbert-Samuel function in the embedded case.
In order to make the latter more concrete, we introduce the following notation.

\begin{definition}[{\cite[Definition~3.1]{CoJS}}]
	\label{Def:normalflat_perm}
	Let $ \mathcal{X} $ be a locally Noetherian scheme and let $ D \subset \mathcal{X} $ be a reduced closed subscheme.
	Let $ I_D \subset \mathcal{O}_{\mathcal{X}} $ be the ideal sheaf of $ D $ in $ \mathcal{X} $.
	\begin{enumerate}
		\item
		The scheme
		$ \mathcal{X} $ is {\em normally flat along $ D $ at $ x \in D $}
		if the stalk $ \gr_{I_D}(\mathcal{O})_x $ of $ \gr_{I_D} (\mathcal{O}) := \bigoplus_{t \geq 0} I_D^t / I_D^{t+1} $ is a flat $ \mathcal{O}_{D,x} $-module.
		Furthermore, $ \mathcal{X} $ is {\em normally flat along $ D $} if $ \mathcal{X} $ is normally flat along $ D $ at every point of $ D $.
		
		\item
		We say that $ D $ is {\em permissible for $ \mathcal{X} $ at $ x \in D $} if the following three conditions hold:
		\begin{enumerate}
			\item $ D $ is regular at $ x $,
			\item $ \mathcal{X} $ is normally flat along $ D $ at $ x $, and
			\item $ D $ contains no irreducible component of $ \mathcal{X} $ containing $ x $.
		\end{enumerate}
		Moreover,
		$ D $ is {\em permissible for $ \mathcal{X} $}, if $ D $ is permissible for $ \mathcal{X} $ at every point of $ D $.
	\end{enumerate}
\end{definition}

\begin{Rk}
	\label{Rk:CJS_Thm_3.3}
	It follows from \cite[Theorem~3.3]{CoJS} that
	if $ D $ is regular and $ y $ is the generic point of the irreducible component of $ D $ containing $ x $,
	then $ \mathcal{X} $ is normally flat along $ D $ at $ x $ if and only if $ H_{\mathcal{X}}(x) = H_{\mathcal{X}}(y) $.
\end{Rk}

The following results, which is included in \cite[Theorem~3.10]{CoJS} shows the close connection between the Hilbert-Samuel function and Hironaka's $ \nu^* $-invariant
from the perspective of resolution of singularities.

\begin{theorem}[{cf.~\cite[Theorem~3.10]{CoJS}}]
	\label{Thm:CJS_3.10}
	Let $ \mathcal X $ be an excellent scheme, or a scheme which is embeddable in a regular scheme $ \mathcal{Z} $.
	Let $ D \subset \mathcal{X} $ be a permissible closed subscheme and let
	$ \pi_{\mathcal{X}} \colon \mathcal{X}' := \Bl_D  (\mathcal{X}) \to \mathcal{X} $,
	resp.~$ \pi_{\mathcal{Z}} \colon \mathcal{Z}' := \Bl_D  (\mathcal{Z}) \to \mathcal{Z} $,
	be the blowing up with center $ D $.
	Take any points $ x \in D $ and $ x' \in \pi_{\mathcal{X}}^{-1} (x) $.
	Then:
	\begin{enumerate}
		\item
		$ H_{\mathcal{X}'} (x') \leq H_{\mathcal{X}} (x) $
		(with respect to the product ordering on $ \N^\N $).
		
		\item
		$ \nu^*_{x'} (\mathcal{X}',\mathcal{Z}') \leq \nu^*_{x} (\mathcal{X},\mathcal{Z}) $
		(with respect to the lexicographical ordering).
		
		\item
		$ H_{\mathcal{X}'} (x') = H_{\mathcal{X}} (x) $
		if and only if
		$ \nu^*_{x'} (\mathcal{X}',\mathcal{Z}') = \nu^*_{x} (\mathcal{X},\mathcal{Z}) $.
	\end{enumerate}
\end{theorem}

\begin{definition}[{\cite[Definition~3.13(1)]{CoJSc}}]
	Let the hypothesis be as in Theorem~\ref{Thm:CJS_3.10}.
	A point $ x' \in \pi_{\mathcal{X}}^{-1}(x) $ is {\em near to $ x $}
	if $  H_{\mathcal{X}'} (x') = H_{\mathcal{X}} (x) $.
\end{definition}

Using the notation of Theorem~\ref{Thm:CJS_3.10},
if $ \mathcal{X} $ is embedded in a regular $ \mathcal{Z} $,
then $ x' \in \pi_{\mathcal{X}}^{-1}(x) $ is near to $ x $ if and only if
$ \nu^*_{x'} (\mathcal{X}',\mathcal{Z}') = \nu^*_{x} (\mathcal{X},\mathcal{Z}) $.		

\medskip

If $ x' \in \pi_{\mathcal{X}}^{-1} (x) $ is not near to $ x $,
the Hilbert-Samuel function (resp.~the $ \nu^\ast $-invariant) detects a strict improvement of the singularity.
Hence, for proving resolution of singularities, it is necessary to find additional invariants, resp.~tools, able to detect an improvement at $ x' $ if the center is chosen suitably.
The directix and the ridge of $ \mathcal{X} $ at $ x $ are objects,
which reveal information on the singularities of the tangent cone $ C_x (\mathcal{X}) $.
They play a crucial role for the task of controlling the locus of near points.

Recall that a polynomial $ f (X) = f(X_1, \ldots, X_n) \in k[X_1, \ldots, X_n] $ is called {\em additive} if $ f(X + Y)  = f(X) + f(Y) $,
where $ (Y) = (Y_1, \ldots, Y_n) $ is a system of indeterminates and
we abbreviate $ (X+Y) := (X_1 + Y_1, \ldots, X_n + Y_n)  $.

\begin{definition}\label{defdirridge}
	\label{Def:Dir_Faite}
	Let $ \mathcal{I} \subset S := k[X_1, \ldots, X_n] $ be an ideal, which is generated by homogeneous elements.
	\begin{enumerate}
		\item
		Let $ \mathcal{T} (\mathcal{I})  \subset \bigoplus_{i=1}^n k X_i $ be the smallest $ k $-vector subspace such that
		\[
			\big( \mathcal{I} \cap k [ \mathcal{T} (\mathcal{I}) ] \big) \cdot  S
			= \mathcal{I} ,
		\]
		where $ k [ \mathcal{T} (\mathcal{I}) ] = \operatorname{Sym}_k(\mathcal{T}(\mathcal{I})) \subseteq S $.
		The {\em directrix of the cone $ \Spec(S/\mathcal{I}) $} is the closed subscheme $ \Dir(S/\mathcal{I}) \subseteq \Spec(S/\mathcal{I}) $ defined by the surjection $ S/\mathcal{I} \to S/ \mathcal{T}(\mathcal{I}) S $.

		\item
		The {\em ridge of the cone $ \Spec(S/\mathcal{I}) $} is the maximal additive subgroup of $ \Spec(S) \cong \mathbb{A}^n_K  $ (considered as
an additive group scheme),
		which leaves the cone $ \Spec(S/\mathcal{I}) $ stable under translation.
	\end{enumerate}

	If $ x \in \mathcal{X} $ is a point of a locally Noetherian scheme $ \mathcal{X} $,
	with an embedding $(\mathcal{X},x)\subset (\mathcal{Z},x)$ for some regular $ \mathcal{Z} $,
	then the directrix $ \Dir_x ( \mathcal{X} ) $ (resp.~ridge $ \Rid_x (\mathcal{X} ) $) of $ \mathcal{X} $ at $ x $ is defined as directrix (resp.~ridge) of the tangent cone $C_x (\mathcal{X})$ of $ \mathcal{X} $ at $ x $, embedded in the Zariski tangent space $T_x(\mathcal{Z})$.
\end{definition}

In fact, one can define the directrix and the ridge of a locally Noetherian scheme $ \mathcal{X} $ at $ x $ without the assumption of an embedding in a regular $ (\mathcal{Z},x) $. Both definitions of directrix and ridge coincide via the embedding $T_x(\mathcal{X})\subset T_x(\mathcal{Z})$.
For details, we refer to \cite[Chapter~2]{CoJS} and \cite[Ch.~I, \S~5)]{Gi1}, respectively.

\medskip

The last notion which we need to recall is Hironaka's characteristic polyhedron \cite{H}, see also \cite[Chapter~18]{CoJS}.
It captures refined information on the singularity of $ \mathcal{X} $ at $ x $,
which is not detected by the tangent cone, the directrix, or the ridge.

\begin{definition}
	\label{Def:CharPoly}
	Let $ (R,\fm,k) $ be a regular local ring and let $ \mathcal{I}  \subset R $ be a non-zero ideal.
	We fix a system $ (u) = (u_1, \ldots,u_e) $ in $ R $, which can be extended to a regular system of parameters for $ R $
	and such that $ \gr_\fm(R)/ \mathcal{T}(\ini_\fm(\mathcal{I} )) \cong k[U_1, \ldots, U_e] $ (using the notation of Definition~\ref{Def:Dir_Faite}),
	where $ U_i := u_i \mod \fm^2 $ is the image of $ u_i $ in the graded ring $ \gr_\fm(R) $.
	\begin{enumerate}
		\item
		Let $ (y) = (y_1, \ldots, y_r) $ be a system of elements in $ R $ such that $ (u,y) $ is a regular system of parameters for $ R $.
		Let $ (f) = (f_1, \ldots, f_m) $ be a standard basis for $ \mathcal{I}  $.
		For every $ i \in \{ 1, \ldots, m \} $, consider expansions $ f_i := \sum_{(A,B)\in\mathbb{Z}^{e+r}_{\geq 0}} C_{A,B,i} u^A y^B $
		with coefficients $ C_{A,B,i} \in R^\times \cup \{ 0\} $.
		The {\em projected polyhedron $\Delta(f;u;y) $ of $ (f) $ with respect to $ (u,y) $}
		is defined as the smallest convex subset of $ \mathbb{R}^e $ which contains all points of the set
		\[
			 \left\{  \frac{A}{\nu_i - |B|} + \mathbb{R}^e_{\geq 0 } \, \bigg| \, C_{A,B,i} \neq 0 \ \wedge \ |B| < \nu_i \right\}
			 ,
		\]
		where $ \nu_i := \ord_\fm (f_i) $, for all $ i \in \{ 1, \ldots, m \} $.
		
		For a vertex $ v \in \Delta(f;u;y) $ and $ i \in \{ 1, \ldots, m \} $,
		the {\em initial form of $ f_i $ at $ v $} is defined as
		\[
			\ini_v (f_i) := \sum_{B : |B| = \nu_i }
			\overline{C_{0,B,i}} \, Y^B + \sum_{(A,B) : \frac{A}{\nu_i - |B|} = v }
			\overline{C_{A,B,i}} \, U^A Y^B \in k[U,Y]
			,
		\]
		where $ \overline{C_{A,B,i}} := C_{A,B,i} \mod \fm  $
		and $ Y_j := y_j \mod \fm^2 $ for $ j \in \{ 1, \ldots, r \} $.
		
		\item
		For fixed $(u)$, the {\em characteristic polyhedron $ \Delta (\mathcal{I} ;u) $ of $ \mathcal{I}  $ at $ \fm $} is defined as the intersection of all projected polyhedra $ \Delta (f;u;y) $, where one varies the choice of $(f;y) $ fulfilling the properties of the first part of the definition,
		\[
			\Delta (\mathcal{I} ;u)
			:= \bigcap_{(f;y)} \Delta(f;u;y).
		\]
		\end{enumerate}
\end{definition}

Given a vertex $ v $ of $ \Delta(g;u;z) $, we say that $ (g;z) $ is {\em prepared at} $v$
if $ (g;u;z) $ is normalized at $ v $ (\cite[Definitions~8.12~and 8.11]{CoJS}) and $v$ is not solvable.
We say that $ (g;z) $ is a {\em suitable choice} if $ (g;u;z) $ is prepared at every vertex $ v $ of $ \Delta(g;u;z) $
(\cite[Definitions~8.11, 8.12~and 8.13]{CoJS}).
The first condition (normalized) is an appropriate choice for the generators of $ {\mathcal{I}}  $,
while the second means that for every vertex $ v \in  \Delta(g;u;z) $ it is impossible to find a change in $ (z) $
such that corresponding projected polyhedron is contained in $  \Delta(g;u;z) $,  $ v $ not being a vertex.

For example, if $ {\mathcal{I}}  $ is the ideal generated by $ f_1 :=  y_1^2 - 2 u_1^2 y +  u_1^4 - u_2^5 $,
then $ \Delta(f_1;u;y_1) $ has two vertices, namely $ v := (2,0) $ and $ w := (0, \frac{5}{2}) $.
It is easy to verify that $ v $ is a solvable vertex:
If we introduce $ z_1 := y_1 - u_1^2 $, then $ f_1 = z_1^2 - u_2^5 $
and the unique vertex $ w $ is not solvable.
In particular, it follows that $ (f_1;z_1) $ is a suitable choice.

By \cite[Theorem~(4.8)]{H}, we have the following equality for a suitable choice $ (g;z) $
\[
	\Delta(g;u;z) = \Delta(\mathcal{I} ;u) ,
\]
and in particular this proves that $\Delta(\mathcal{I} ;u)$ is actually a polyhedron.
In \cite{H}, Hironaka proved that at least in the $ \fm $-adic completion $ \widehat{R} $ of $ R $, there exists a suitable choice $ (\widehat{g}, \widehat{u}) $ for $ (\mathcal{I} \widehat{R}; u) $ such that $ \Delta(\widehat{g}; u ; \widehat{z}) = \Delta (\mathcal{I} ;u) $.
In general, it is not clear, whether there exists a suitable choice without passing to the completion.
It is shown in \cite{CoP,CoS} that Hironaka's result holds without passing to the completion if we assume $ R $ to be excellent and if we additionally require mild technical conditions, which are fulfilled in many cases, e.g., if the residue field of $ R $ is perfect, or if $ R $ is Henselian, or in polynomial situations.

\medskip

While the definition of the characteristic polyhedron depends on an embedding,
it is still a useful source for invariants of the singularity of $ \mathcal{X} $ at $ x $.
For example, the number $ \delta (\Delta(\mathcal{I} ;u)) $, defined just below, is actually an invariant of $ \Spec(R/\mathcal{I} ) $.
This topic has been investigated in great details in \cite{CoJSc}.

\begin{definition}
	\label{Def:delta_first}
	Let $ \Delta \subseteq \mathbb{R}^e_{\geq 0 } $ be a non-empty, closed, convex subset such that $ \Delta + \mathbb{R}_{\geq 0}^e = \Delta $,
	where $ + $ denotes the Minkowski sum.
	Set
	\[
		\delta (\Delta) := \min \{ v_1 + \cdots + v_e \mid v = (v_1, \cdots, v_e) \in \Delta  \}.
	\]
	We define the {\em first face of $ \Delta $} as the face of $ \Delta $ consisting of all points $ v  = (v_1, \ldots, v_e) \in \Delta $ with $ v_1 + \cdots + v_e  = \delta (\Delta) $.
\end{definition}

\section{Giraud's presentations and their application to cones}

To prove Theorem~\ref{th:th}, we need Giraud's Theory of presentations \cite{Gi2}.
In particular, we require Theorem~\ref{th:HSridgecone} below, which is a refinement of \cite[Proposition~4.3]{Gi2}.
A {\em presentation $ \cP $} consists of the following data
\[
(S; R;
z_1, \ldots,z_e;
\mathcal{I} ;
f_1,\ldots,f_m;
n_1,\ldots,n_m;
E;
P_1,\ldots,P_e;
q_1,\ldots,q_e;
s_1,\ldots,s_e)
,
\]
where:

\smallskip

\noindent
(a) $ S \subseteq  R $ are regular local rings with respective maximal ideals $ N $ and $ M $,  such that
the residue field extension is trivial, say $ k := R/M = S/N $,
and such that the natural morphism of graded rings $ \gr_N(S) \to \gr_M (R) $ is flat.

Further, $ (z_1, \ldots, z_e ) $ are differential local coordinates of $ R/S $ (in the sense of \cite[Définition~2.2(iii)]{Gi2}).

\smallskip

\noindent
(b) $ \mathcal{I}  \subset R $ is an ideal such that $ \gr_M(R/\mathcal{I} ) $ is flat over $ \gr_N(S) $.

\smallskip

\noindent
(c)
$ f_1, \ldots, f_m  $ are elements in $ \mathcal{I}  $ such that their images in $ R_0 := R/NR $ form a standard basis\footnote{To ask that it is a standard basis is slightly more restrictive than Giraud's original definition, but it is important for our setting.} for the ideal $ \mathcal{I} _0 := \mathcal{I} R_0 $
and
$ n_1, \ldots,  n_m $ are positive integers such that
$ \nu^*_{M_0} (\mathcal{I} _0, R_0) = (n_1, \ldots, n_m, \infty, \ldots ) $,
for $ M_0 := M R_0 $.

Moreover, $ E \subseteq \mathbb{Z}_{\geq 0}^e $ is a certain subset fulfilling $ E + \mathbb{Z}_{\geq 0}^e  = E $
(the so-called exponents of the ideal $ \mathcal{I} _0 $ with respect to $ (z) $)
such that for all $ i \in \{ 1, \ldots, m \} $ and $ A = (A_1, \ldots, A_e) \in E $ with $ |A| = A_1 + \cdots + A_e < n_i $, we have
$ D_A^{(z)} f_i \equiv 0 $.
Here, $ D_A^{(z)} f_i $ denotes the Hasse-Schmidt derivative of $ f_i = f_i (z)$ with respect to $ z^A $,
meaning the coefficient of $ T^A $ in the Taylor expansion of $ f_i(z+T) = f_i (z_1+T_1, \ldots, z_e+ T_e) $ with respect to $ T = (T_1, \ldots, T_e) $ (when it does exist, e.g., in the completion of $ R $).

\smallskip

\noindent
(d)
$ P_1, \ldots, P_e \in S[X_{A,i} \mid A \in \mathbb{Z}_{\geq 0}^e, i \in \{ 1, \ldots, m \} : |A| < n_i ] $ are polynomials with coefficients in $ S $ that are homogeneous of degree $ q_i $ if we assign to $ X_{A,i} $ the degree $ n_i - |A| $.

Finally, $ s_1, \ldots, s_e \in R $ are elements in $ R $ such that $ s_j = P_j (D_A^{(z)} f_i) $ for all $ j \in \{ 1, \ldots, e \} $.
Additionally, there is also a condition on the Hilbert-Samuel series of $ R_0/ (s) R_0  $, which will be automatically fulfilled in our setting as the elements $ s_i $ will form a triangular system.
	Thus, we do not recall the details here.

\medskip

For the precise definition of a presentation we refer to \cite[D\'efinition~3.1]{Gi2}.

\begin{theorem}\label{th:HSridgecone}
	Let $\cP $ be a presentation as recalled just before.
	Suppose that $ S $ is excellent and contains a field.
	Let $ x \in \mathcal{X} $ be the closed point corresponding to the maximal ideal $ M \subset R $.
		Using the notation of the definition, we further set
		\[
			\mathcal{X} := \Spec (R/\mathcal{I} )
			\ \ \
			\mathcal{X}_i := \Spec (R/ (f_i) ),
			\ \ \mbox{ and } \ \
			\mathcal{Y}_j := \Spec (R/(s_j)),
		\]
		for $ i \in \{ 1, \ldots, m \} $ and $ j \in \{ 1, \ldots, e \} $.
		The Hilbert-Samuel stratum of $ \mathcal{X} $ at $ x $ fulfills the following equality
	\[
	\HS_{\mathcal{X},x}\ = \bigcap_{i=1} ^m \HS_{\mathcal{X}_i,x}\ \subset \bigcap_{j=1} ^e \HS_{\mathcal{Y}_j,x},
	\]
	where we abbreviate $ \HS_{\mathcal{X},x} := \mathcal{X} ( H_{\mathcal{X}} (x) ) $ (Definition~\ref{Def:HS}) and analogously for $ \HS_{\mathcal{X}_i,x} $ and $ \HS_{\mathcal{Y}_j,x} $.
\end{theorem}

\begin{proof}
	Let $\mathcal H  $ be a reduced irreducible subscheme of $ \mathcal{Z} := \Spec (R) $.
	In the case that $ \mathcal H $ is regular, we have by \cite[Proposition~4.3]{Gi2},
\begin{equation}\label{eq:HSridgecone}
\begin{array}{c}
\mathcal{H} \subseteq \HS_{\mathcal{X},x}
\ \  \iff \ \
\displaystyle \mathcal{H}\subseteq \bigcap_{i=1}^m \HS_{\mathcal{X}_i,x}
\\
\displaystyle \mathrm{and}\ \mathcal{H}\subseteq \bigcap_{j=1} ^e \HS_{\mathcal{Y}_j,x},
\ \  \mbox{if} \ \
\mathcal{H} \subseteq \HS_{\mathcal{X},x}
\end{array}
\end{equation}
Suppose that $ \mathcal H $ is singular at $ x $.
We choose a curve $\Gamma$ such that  $x\in \Gamma \subset {\mathcal H}$, $\Gamma \not\subset \Sing(\mathcal H)$.
We perform a sequence of blowing ups
\begin{equation}\label{eq:bups}
\mathcal{Z}=:\mathcal{Z}_0 \stackrel{\pi_1}{\longleftarrow} \mathcal{Z}_1 \stackrel{\pi_2}{\longleftarrow} \cdots \stackrel{\pi_n}{\longleftarrow} \mathcal{Z}_n {\longleftarrow} \ldots,
\end{equation}
where $ \pi_1 \colon \mathcal{Z}_1 \to \mathcal{Z}_0 $ is the blowing up with center $ x $, and for $ i \in \{ 1, \ldots, n-1 \} $, $ \pi_{i+1} \colon \mathcal{Z}_{i+1} \to \mathcal{Z}_{i} $ is the blowing up with center some closed point
$x_i \in \mathcal{Z}_i$, exceptional for $ \pi_{i}$ and on the strict transform  $\Gamma_i$ of $\Gamma$.
By \cite[(2.2.3) p.~71]{B}, there exists $N\in \N$ such that, for $n\geq N$, $\nu^*_{x_n} ( {\mathcal H}_n, \mathcal{Z}_n )= \nu^*_{x_N} ( \mathcal{H}_N, \mathcal{Z}_N )$,
where $\mathcal{H}_n$ (resp. $\mathcal{H}_N$) is the strict transform of $\mathcal{H}$ and $\mathcal{Z}_n$ (resp. $\mathcal{Z}_N$) of $\mathcal{Z}:=\Spec (R)$.
Of course, for $N$ big enough,  $\Gamma_n$ is regular at $x_n$ for $n\geq N$.
Then, by \cite[Proposition~(3.1) p.~74]{B}, $\Gamma_n$ is permissible for
$\mathcal{H}_n$ at $x_n$: the Hilbert-Samuel function of  $\mathcal{H}_n$ is constant along $\Gamma_n$ in a neighbourhood of $x_n$, so $\mathcal{H}_n$  is regular at $x_n$.
Furthermore, $ n $ is defined to be any natural number such that the strict transform $\mathcal{H}_n$  of $\mathcal{H}$ is regular at $x_n$.

Notice that $ \mathcal{H}_n $ is an irreducible component of the Hilbert-Samuel stratum of the strict transform of $ \mathcal{X} $ (resp.~$ \mathcal{X}_i$, resp.~$ \mathcal{Y}_j $)
		at $ x_n $
		in $ \mathcal{Z}_n $ if and only if $ \mathcal{H} $ is an  irreducible component of the Hilbert-Samuel stratum of $ \mathcal{X} $ (resp.~$ \mathcal{X}_i$, resp.~$ \mathcal{Y}_j $)
		at $ x $.
Let $\mathcal{P}_n$ be the transform of $\mathcal{P}$, i.e., the presentation defined by an induction of \cite[Theorem~5.2]{Gi2}
	(which requires the assumption that $ S $ contains a field).
Since $\mathcal{H}_n$ is regular at $x_n$,
we may apply \cite[Proposition~4.3]{Gi2} locally at $ x_n $.
Therefore, \eqref{eq:HSridgecone} holds for $ \mathcal{H}_n $ (with the respective transforms of $ \mathcal{X}, \mathcal{X}_i $, and $ \mathcal{Y} $) and thus the statements also hold for $ \mathcal{H} $.
This proves the assertion.	
\end{proof}

Next, we apply Theorem~\ref{th:HSridgecone} in the setting of cones.
Let us fix the notation:
Let $ k $ be a field  and $\mathcal{I}\subset k[X_1, \ldots,X_n]$ be an ideal generated by homogeneous polynomials.
	We set $C := V (\mathcal{I}) $.
	Let $J, \mathcal{J}$ be the ideals of the directrix and of the ridge of $\mathcal{I}$ respectively (Definition \ref{defdirridge}). Then by \cite[Section~1.5, Lemma~1.6]{Gi2} or \cite[Proposition~5.4 page I-27]{Gi1}, up to renumbering the variables, there exist $(F_1,\ldots,F_m)$ homogeneous generators of $\mathcal{I}$ of degree $n_i=\deg(F_i),\ 1\leq i \leq m,$  with
\begin{equation}\label{eq:cone}
D_A^{(X)}F_i \equiv 0 \ \mathrm{for}\ A\in \exp(\mathcal{I} )=:E,\ i \in \{ 1, \ldots, m \} \mbox{ with } \vert A\vert<\deg F_i,
\end{equation}
where $ \exp(\mathcal{I} ) := \{  \exp(h) \mid  h \in \mathcal{I} \setminus \{ 0 \} \mbox{ homogeneous} \}$
and $ \exp(h) $ denotes the dominant exponent of the polynomial $ h $ with respect to the lexicographical ordering on $ \N^n $.
Furthermore $\mathcal{J}$ is generated by a triangular basis of additive homogeneous polynomials $(s_1,\ldots,s_e)$ with
\begin{equation}\label{eq:faite}
s_i=X_i^{q_i}+\sum_{j>i}c_{i,j}X_j^{q_i},\ c_{i,j}\in k,\ 1\leq i \leq e,
\mbox{ with } q_1\leq q_2\leq \cdots\leq q_e.
\end{equation}
By \cite[Section~1.5, Lemma~1.7]{Gi2} or \cite[Proposition~5.4 page I-27]{Gi1}, there are polynomials $P_1,\ldots,P_e$, $P_i\in k[X_{A,j}  \mid 1\leq j \leq m$, $\vert A \vert <n_{i}]$, $1\leq i \leq m$, homogeneous of degree $q_i$ when we give to  $X_{A,j}$ the degree $n_i - \vert A\vert$,
such that
\begin{equation}\label{eq:P}
s_j=P_j(D_A^{(Z)} F_i).
\end{equation}

To draw the connection with the notion of presentation of \cite{Gi2}, we denote
\[
\begin{array}{l}
Z_i:=X_i, \ 1\leq i \leq e, \ R:=k[Z_1, \ldots,Z_e,X_{e+1},\ldots,X_n]_M,\ M:=(X_1, \ldots,X_n),
\\
 S:=k[X_{e+1}, \ldots, X_m]_N, N:=(X_{e+1}, \ldots, X_m).
\end{array}
\]
The reader will verify that
\begin{prop}\label{pr:revelation}
\begin{equation}\label{eq:Pr}
\begin{array}{rcl}
	\mathcal{P} & := &
	(S;R;Z_1, \ldots,Z_e;\mathcal{I};F_1,\ldots,F_m;n_1,\ldots,n_m;E;
	\\ && \hspace{3cm}  P_1,\ldots,P_e;q_1,\ldots,q_e;s_1,\ldots,s_e)
\end{array}
\end{equation}
is a presentation as defined by Giraud.
\end{prop}

By \cite[Lemme~1.7]{Gi1}, this presentation has the following supplementary property:
\begin{equation} \label{eq:supp}
F_i\in k[s_1,\ldots,s_e],\ 1\leq i \leq m.
\end{equation}

\begin{Rk} \label{Rk:HSridge}
	By Theorem~\ref{th:HSridgecone} and \eqref{eq:supp}, the Hilbert Samuel strata of the origin of the cone $ C $ and of its ridge locally coincide.
\end{Rk}

The word ``locally" may be skipped, indeed these strata coincide in the full affine scheme $\Spec k[X_1,\ldots,X_n]$, as stated in the proposition below
which makes more precise  M.J.~Pomerol's Theorem \cite[Proposition~2.2]{P}.

\begin{prop}\label{Pr: HSconeridge}
	With the notations of the beginning of this section, the Hilbert Samuel strata of the origin of the cone $C$ and of its ridge coincide in $\Spec  k[X_1, \ldots,X_n]$.
	Their ideals are respectively the reduction of
$$(D^{(X)}_A F_i\ ; \ \vert A \vert < n_i,\ 1\leq i \leq m)$$
and  the reduction of
$$(D^{(X)}_A s_i\ ; \ \vert A \vert < q_i,\ 1\leq i \leq e)$$
which both coincide.
\end{prop}
\begin{proof}
	Denote by $\mathcal R$ the ridge of $ C $.
Let us blow up the origin. The strict transforms of $C$  and $\mathcal R$ are the tautological line bundles over
$\Proj k[X]/\mathcal{I}$ and $\Proj k[X]/\mathcal{J}$ respectively. So the strict transforms of the Hilbert Samuel strata are empty or line bundles over Hilbert Samuel strata of $\Proj k[X]/\mathcal{I}$ and $\Proj k[X]/\mathcal{J}$ respectively: these Hilbert Samuel strata are subcones
of $\Spec  k[X_1, \ldots,X_n]$ and by Remark~\ref{Rk:HSridge}, they coincide.
	The equalities of the ideals are consequences of Theorem~\ref{th:HSridgecone}.
\end{proof}

\begin{prop} \label{prop:HSconeconstant}
Let $ k $ be a field  and $\mathcal{I}\subset k[X_1, \ldots,X_n]$ be an ideal generated by  homogeneous polynomials, we set $C := V (\mathcal{I}) $. Let $J, \mathcal{J}$ be, respectively, the ideals of the directrix and of the ridge of $\mathcal{I}$. The Hilbert-Samuel function is constant on $ C $ if and only if $ \mathcal{I}_{red}=\mathcal{J}_{red}=J $.
\end{prop}

\begin{proof}

By Proposition~\ref{Pr: HSconeridge}  above, it is enough to prove the statement when  $\mathcal{I}=\mathcal{J}$.

Let us note that, when $\car(k)=0$, additive polynomials are homogeneous of degree~1. Hence, $J=\mathcal{I}$ and $C$ is the intersection of hyperplanes of $\Spec{k[X_1, \ldots ,X_n]}$,
which implies the proposition in this easy case.

\smallskip

Now we consider the case  $\car(k)=p>0$. As seen above ({\it viz.} before stating Proposition \ref{pr:revelation}),
 there exist variables $(Z_1,\ldots,Z_e; W_1,\ldots,W_d)$ in $k[X_1,\ldots,X_n]$ ($d+e=n$) such that
$\mathcal{I}=(\sigma_1,\ldots,\sigma_e)$, where $\sigma_i$ are homogeneous additive polynomials of degrees $q_i$ with $q_1\leq q_2\leq \cdots \leq q_e$ and
$$\sigma_i= Z_i^{q_i}+t_i(Z_{i+1},\ldots,Z_e,W_1,\ldots,W_d).$$

When $\mathcal{I}_{red}=J$, by definition of the directrix, after translations on  $(Z_1,\ldots,Z_e)$ if necessary, we get
$C_{red}=V(Z_1,\ldots,Z_e)$.
Furthermore, for any
point $x\in C$, we have $\nu_x^*(C,\A^n)=(q_1,\ldots,q_e,\infty,\infty,\ldots)$,  where $\A^n_k:=\Spec k[Z_1,\ldots,Z_e; W_1,\ldots,W_d]$ (Definition~\ref{Def:std_basis_nu_ast}(2)).
This implies that $C$ is normally flat over $C_{red}$ (\cite[Theorem~3.2(2)]{CoJS}), which is equivalent to the constancy of the Hilbert-Samuel function on  $C$ (\cite[Theorem~3.3]{CoJS}). This proves the  converse implication in the proposition.

Let us prove the direct implication. We are in the very extreme case where, in the corresponding presentation \eqref{eq:Pr}, $e=m,F_i=s_i$, $1\leq i \leq e$.
	Then, by Theorem~\ref{th:HSridgecone}, $\bigcap_{i=1} ^m \HS_{\mathcal{C}_i,x}= \HS_{\mathcal{C},x}=C_{red}$.
Consider the additive group subscheme $ D \subset \A^n$ defined by the equations $\sigma_i^{q_e/q_i}, 1 \leq i \leq e$.
Then  $D $  is  flat over $\A^d=\Spec k[W_1,\ldots,W_d]$. Note that $D_{red}=C_{red}$.

As $ D $ is a cone, Theorem \ref{th:HSridgecone} applied to $D$ implies that $ \HS_{D,x} = D_{red} $.
	In other terms,
it can be assumed that $q_1 = \cdots  = q_e$.
In this case, we may replace the $\sigma_i$ by some  $\tau_i = Z_i^{q_e}+ r_i$, $r_i\in k[W_1,\ldots,W_d]$  for $1 \leq i \leq e$
by performing a
linear change of coordinates.
 Since $\HS_{D,x} = D_{red} $ and since the natural morphism $\eta \colon D \longrightarrow \A^d$
is dominant, each polynomial $s_i$ has a prime factor of order $ q_e$ as an element of $K[Z_i]$, where $K = \Frac(k[W_1,\ldots,W_d])$.
 Therefore $r_i \in K^{q_e}$. Now $k[W_1,\ldots,W_d]$ is integrally closed, so $r_i \in k[W_1,\ldots,W_d]^{q_e}$. Up to an affine $k[W_1,\ldots,W_d]$-linear change of coordinates on $\A^n$, we thus have $\tau_i =Z_i^{q_e}$, $1\leq i\leq e$, so the conclusion holds.
\end{proof}

\section{Proof of the Theorem}

Recall that our goal is to prove that the Hilbert-Samuel function is locally constant on a Noetherian  scheme $ \mathcal{X} $
such that $\mathcal{O}_{\mathcal{X},x}$ is excellent for every $x\in \mathcal{X}$
if and only if its reduction $ \mathcal{X}_{red} $ is regular and $ \mathcal{X} $ is normally flat along $ \mathcal{X}_{red} $.
By \cite[Theorem~3.3]{CoJS} (see Remark~\ref{Rk:CJS_Thm_3.3}),
we have only need to prove the following~: if the Hilbert-Samuel function is  constant  on $\mathcal{X}$,  then  $\mathcal{X}_{red}$ is everywhere regular.

%



%
%
%

%
%
\begin{lem} \label{lem:lem3}	
Let $ \mathcal {X}=\Spec A $ with $A$ a catenary Noetherian local ring,	
 let  $x\in\mathcal X$ be the closed point.
 Assume that the Hilbert-Samuel function is constant on $ \mathcal {X}$. Let  $ \Dir_x (\mathcal{X})\subset \Rid_x (\mathcal{X}) \subset T_x(\mathcal{X})$ be respectively the directrix and the ridge of the tangent cone of $\mathcal{X}$ at $x$, embedded in the Zariski tangent space $T_x(\mathcal{X})$
 (Definition \ref{defdirridge} and following comments). Then, we have
 $$ \Dir_x (\mathcal{X}) = (\Rid_x (\mathcal{X}))_{red}.$$
 \end{lem}

 \begin{proof}

	When  $ \dim \mathcal{X}=0$, $ \Dir_x (\mathcal{X}) = \Rid_x (\mathcal{X})_{red}$ is the origin in $T_x(\mathcal{X})$.
	From now on, we assume $ \dim \mathcal{X} \geq 1 $.

Let us note the following remark that we will use later on.
\begin{Rk}\label{Rk:semicon}
Let $\mathcal{X}'\longrightarrow  \mathcal{X}$ be the blowing up centered at  $x$. The common value  of the Hilbert-Samuel function at the generic points of the irreducible components of  $\mathcal{X}$ do not change: it is $H_{\mathcal{X}}(x)$. Let  $x'\in \mathcal{X}'$ map to  $x$. By specialization \cite[Theorem 2.33(1)]{CoJS}, $H_{\mathcal{X}'}(x')\geq H_{\mathcal{X}}(x)$ and by \cite[Theorem 3.8 p. II.28]{Gi1},   $H_{\mathcal{X}'}(x')\leq H_{\mathcal{X}}(x)$: every point $x'$ above $x$ is near to $x$.
	\end{Rk}

	 \begin{lem}\label{lem:lem2}
	 	Let $ k $ be a field and $ \mathcal{I}\subset R:=k[X_1,\ldots,X_n]$ be a homogeneous ideal.
 	Let $\mathcal{J}$ be the ideal of its ridge and $J$ the ideal of its directrix. Assume
 $$J\not=\mathcal{J}_{red}.$$
Let $x\in \Spec (R/ \mathcal{I})$ be the origin and $\mathcal{X}' \longrightarrow \Spec (R/ \mathcal{I})$ be the blowing up along $x$.
There exists $ x' \in \mathcal{X}'$ above $x$ such that  $$H_{\mathcal{X}'} (x') <H_{\Spec (R/ \mathcal{I})}(x).$$
\end{lem}

\begin{proof}
	By \cite[Lemme~5.2.2]{Gi2}
	every point in $\mathcal{X}'$ is on the strict transform of the ridge and is near to $x$ as a point of the ridge.  	
	Furthermore, a point $x'$ near to $x$ is on the strict transforms of $ \mathcal{Y}_1,\ldots, \mathcal{Y}_e $,
	the hypersurfaces of equations $\sigma_i$ of degrees $q_1, \ldots ,q_e$, $q_1\leq  \cdots \leq q_e$,
$\mathcal{J}=(\sigma_1, \ldots , \sigma_e)$ and near to $x$ for all $ \mathcal{Y}_i $.

\smallskip

	Let us define, as in the proof of Proposition~\ref{prop:HSconeconstant}, $\tau_i=Z_i^{q^e}+r_{i}$, $r_i\in k[W_1,\ldots,W_c]$,
where $\mathrm{Vect}_k(X_1, \ldots ,X_n)=\mathrm{Vect}_k(Z_1, \ldots ,Z_e, W_1, \ldots ,W_c))$ is a renaming of variables after
performing a $k$-linear change of variables.
	As
	 $J\not=\mathcal{J}_{red}$, there is at least one $r_{i_0}$ which is not a $q_e$-th power.
	 So, there is a point $x'\in \pi^{-1}(x)$ on the strict transform of the (unique) prime factor of $\tau_{i_0}=Z_{i_0}^{q^e}+r_{{i_0}}$ which is not near to $x$ for the hypersurface  of equation $\tau_ {i_0}$.
	 Hence, it is not near to $x$ for some $\mathcal{Y}_i$:
	 this point $x'$ is not near to  $x$ for $\mathcal{X}$.
\end{proof}

\noindent
	{\bf \em End of the proof of Lemma~\ref{lem:lem3}.}
	From now on, we suppose that $ \Dir_x (\mathcal{X}) \subsetneq \Rid_x (\mathcal{X})_{red}$ and we will deduce a contradiction.
	By going to the completion, we may suppose that $ \mathcal{O}_{\mathcal{X},x} $ is the quotient of a regular ring $\mathcal{O}_{\mathcal{Z} ,x}$ of residue  field $k$.
	Let $(u,y):=(u_1,\ldots,u_d,y_1,\ldots,y_r)$ be a regular system of parameters of $\mathcal{O}_{\mathcal{Z} ,x}$ such that
	the initial forms of $ (y) $ (with respect to the maximal ideal $\fm$ at $ x $) are a standard basis for	
	the ideal of the directrix of $\mathcal{X}$ at $x$ embedded in $\Spec k[U,Y]$,
where  $ U_i := \ini_\fm (u_i) $ and $ Y_j := \ini_\fm (y_j) $,
for $ i \in \{ 1, \ldots, d \}, j \in \{ 1, \ldots , r \} $.
 Let  $(f_1,\ldots,f_m)$ be a standard basis of $\mathcal{I}\subset \mathcal{O}_{\mathcal{Z} ,x}$ with respect to $(u,y)$.
Let  $\mathcal{Z}'\longrightarrow \mathcal{Z} $ be the blowing up along $x$.
	Denote by $\mathcal{X}'$ the strict transform of $\mathcal{X}$ and recall that $\mathcal{I}\subset \mathcal{O}_{\mathcal{Z},x}$ is the ideal of $\mathcal{X}$.
	The fibre above $x$ is canonically isomorphic  to the  Proj of the tangent cone of $ \mathcal X $ at $ x $.
\begin{equation}\label{eq:x'pasproche}
{\hbox{Let $x'$ be a point in this fibre not near to $x$ for the tangent cone. }}
\end{equation}
Let $\mathcal{I}' \subset \mathcal{O}_{\mathcal{Z}',x'} $ be the strict transform of $\mathcal{I}$ and $t\in \mathcal{O}_{\mathcal{Z}',x'}$ be a generator of the exceptional ideal.
	By \eqref{eq:x'pasproche},  there exists a standard basis $(g_1,\ldots,g_{m'})\in (\mathcal{O}_{\mathcal{Z}',x'}/(t))^{m'}$ of $\mathcal{I}'   \mod(t)$ with $\nu^*_{x'} (g_1,\ldots,g_{m'})<_{lex}\nu^*_{x}( \mathcal{I})$ (notation \eqref{eq:nu*}), this implies $\nu^*_{x'} (\mathcal{I}')<_{lex}\nu^*_{x}( \mathcal{I})$:  $x'$ is not near to $x$, this is the contradiction.
\end{proof}

 \begin{lem} \label{lem:lem4}
  Let $ \mathcal {X}=\Spec R/\mathcal{I} $, with $R$ an excellent regular local ring, and let $x\in\mathcal X$ be the closed point,
  $\mathcal{Z}:=\Spec R$.
  Assume that the Hilbert-Samuel function is constant on $ \mathcal {X}$. 	
  Then, for any ``adapted choice of variables", the characteristic polyhedron at $x$ is empty (Definition \ref{Def:CharPoly}).
 \end{lem}

\begin{proof}
	 Let us first precise what is an ``adapted choice of variables": we mean a system $(u):=(u_1,\ldots,u_d)\in \mathcal{O}_{\mathcal{Z},x}^d$ which can be
extended to a regular system of parameters $(u,y)=(u_1,\ldots,u_d,y_1,\ldots,y_r)$ of $R$ such that $(\ini_\fm(y_1),\ldots,\ini_\fm(y_r))$ are equations of the directrix of $\mathcal{X}$ at $x$,
	 where $ \fm \subset  R $ is the unique maximal ideal.
	 The statement is that for any such system $(u)$, $\Delta(\mathcal{I};u)=\emptyset$ where $\mathcal{I}\subset\mathcal{O}_{\mathcal{Z},x} $ is the ideal of $\mathcal{X}$.

  Suppose the statement is wrong.
  Thus, we can find a regular system of parameters $(u,y)=(u_1,\ldots,u_d,y_1,\ldots,y_r)$ and  a standard basis $f=(f_1,\ldots,f_m)$  at $x$ of the ideal ${\mathcal{I}}$ such that the vertices of the first face of $\Delta(f;u;y)$ (Definition~\ref{Def:delta_first}) are ``prepared"
  (see comments right after Definition \ref{Def:CharPoly}). This implies that they are vertices of the characteristic polyhedron  $\Delta(\mathcal{I};u)$.
  Let $ \pi_1 \colon \mathcal{Z}'\longrightarrow  \mathcal{Z}$ be the blowing up along $x$ and let $x'$  be the point of parameters $(u',y'):=(u_1,u_2/u_1,\ldots,u_d/u_1,y_1/u_1,\ldots,y_r/u_1)$ (origin of the ``$ u_1 $-chart'').
  Set $\delta := \delta(\Delta(f;u;y)) =\delta(\Delta(\mathcal{I};u))$ (Definition~\ref{Def:delta_first})
  to be the modulus of the vertices of the first face.

  \smallskip

  As $\ini_\fm(f_i)\in k(x)[\ini_\fm(y_1),\ldots,\ini_\fm(y_r)]$, $1\leq i \leq m$, where $ k := R/ \fm $, we have $ \delta > 1 $.
  The usual computations (e.g., analogous to \cite[Proof of Proposition~3.15, formula (4.4)]{CoScDim2}) show that the smallest first
   coordinate of points of $\Delta(f';u';y')$ is $\delta-1>0$.
  Moreover, a point $v'$ with first coordinate $\delta-1$ is obtained by the affine transformation
  $$(a_1,a_2,\ldots,a_d) \mapsto (a_1+\cdots+a_{d}-1,a_2,\ldots,a_d)=(\delta-1,a_2,\ldots,a_d)
  $$
  where $ v := (a_1,a_2,\ldots,a_d)$ is a point of the first face of $\Delta(f;u;y)$.
  Note that if $v'$ is a vertex of $\Delta(f',u',y')$, then $v$ is a vertex of the first face of $\Delta(f,u,y)$.

  \smallskip

  For $ j \in \{ 1, \ldots, m \} $ and $v$ a vertex,
  we write the initial form of $ f_j $ at $ v $ (Definition~\ref{Def:delta_first}) as $$\ini_v(f_j)=F_j(Y_1,\ldots,Y_r )+\sum_{0 \leq \vert A\vert \leq m_j-1} \lambda_A Y^A U^{(m_j-\vert A\vert)v},$$
  where $F_j\in k[Y_1,\ldots,Y_r]$ is homogeneous of degree $m_j:=\ord_{\fm} (f_j)$.
  If the corresponding point $ v' := (\delta - 1 , a_2 , \ldots, a_d ) \in \Delta (f';u';y') $ after the blowing up is a vertex,
  we have
  \begin{equation}\label{eq.ini}
  \begin{array}{l}
  \ini_{v'}(f'_j) =
  \\[8pt]
  \displaystyle

  F_j(Y') + \sum_{0 \leq \vert A\vert\leq m_j-1} \lambda_A {Y'}^A {U'_1}^{(m_j-\vert A\vert)(\vert v\vert -1)} {U'_2}^{(m_j-\vert A\vert)a_2}\cdots {U'_d}^{(m_j-\vert A\vert)a_d }=
  \\[15pt]
  \displaystyle

   F_j(Y') + \sum_{0\leq \vert A\vert\leq m_j-1} \lambda_A {Y'}^A {U'_1}^{(m_j-\vert A\vert)(\delta -1)} {U'_2}^{(m_j-\vert A\vert)a_2}\cdots {U'_d}^{(m_j-\vert A\vert)a_d } .
   \end{array}
    \end{equation}

   Further, since $\ini_v(f_j)$ is prepared, so is $ \ini_{v'}(f'_j)$.
   In particular, $v'$ is not solvable, i.e.: $v'$ is a vertex of $\Delta(\mathcal{I}';u')$ where $\mathcal{I}'$ is the strict transform of $\mathcal{I}$.
   By Remark~\ref{Rk:semicon}, $V(y',u_1)\subset \mathcal{X}'$ is permissible at $x'$, id est  $\delta-1\geq 1$.

\begin{claim}\label{cl:delta}  If  $\delta-1= 1$, we claim that the ideal
$$
J_\zeta \subset \mathrm{gr}_\zeta \mathcal{O}_{\mathcal{Z}',\zeta}=
\mathrm{Fr}\left ({\mathcal{O}_{\mathcal{Z}',\zeta} \over (u'_1,y')}\right ) [U'_1, Y'_1, \ldots ,Y'_r]
$$
of the directrix at the generic point $\zeta$ of $V(y',u_1)$ is generated by the initial forms $(U'_1,Y'_1,\ldots,Y'_r)$ of the elements $(u'_1:=u_1,y'_1,\ldots,y'_r)$ (abuse of notations with \eqref{eq.ini}).
  \end{claim}
  \begin{proof}
   For $1\leq j\leq m$,
  the initial forms are (with the obvious abuse of notation)
\begin{equation}\label{lasteq}
 \ini_{\zeta}(f_j)= F_j(Y'_1,\ldots,Y'_r )+\sum_{0 \leq \vert A\vert\leq m_j-1} F_{j,A} {Y'}^A {U'_1}^{(m_j-\vert A\vert)(\delta -1)}
\end{equation}
$$\in \mathcal{O}_{\mathcal{Z}',x'}/ (u_1', y')[U'_1,Y_1', \ldots, Y_r'] , $$
by identifying  $\mathcal{O}_{\mathcal{Z}',x'}/ (u_1', y')$ with the polynomial ring $k[\overline{u}_2',\ldots,\overline{u}_d']_{(\overline{u}_2',\ldots,\overline{u}_d')}$,
where $ \overline{u}_j'=u_j'\ \mod (u_1', y')$, $2\leq j \leq d$, we get $F_{j,A} \in k[\overline{u}_2',\ldots,\overline{u}_d']$.

By \cite[Lemma~1.9]{H2},  $(\ini_{\zeta}(f_1),\cdots,\ini_{\zeta}(f_m))$ generate in $k(\zeta)[U'_1,Y_1', \ldots, Y_r']$ the ideal of the tangent cone of
$\mathcal{X}'$ at $\zeta$.
The field extension $k(x)\longrightarrow k(\zeta)=k(x)(\bar{u_2}',\ldots,\bar{u_d}')$ is separable, so
by \cite[Lemma 2.10 p.18]{CoJS}, the ideal of the directrix of $(F_1(Y'_1, \ldots ,Y'_r),\ldots,F_m(Y'_1, \ldots ,Y'_r))$ in $k(\zeta)[U'_1,Y_1', \ldots, Y_r'] $ is $(Y'_1,\ldots,Y'_r)$.
Assume   that the ideal $J_\zeta \subset k(\zeta)[U'_1,Y_1', \ldots, Y_r']$
 has codimension   $\leq r$.
 	Then, we have   $J_\zeta= (Z'_1,\ldots,Z'_r)$, with
\begin{equation}\label{eq.tauzeta}
 Z'_i=Y'_i+\lambda_i U'_1,\ \lambda_i\in k(\zeta), \ 1\leq i \leq r.
 \end{equation}
 Furthermore, we have, for $ 1\leq j \leq m$,
\begin{equation}\label{eq.zeta}
\ini_{\zeta}(f_j)= F_j(Y'_1+\lambda_1  U'_1,\ldots,Y'_r+\lambda_r U'_1 ) \in\mathcal{O}_{\mathcal{Z}',x'}/ (u_1', y')[U'_1,Y_1', \ldots, Y_r'] .
 \end{equation}

 Take any valuation $w$ on $k(\zeta)[U'_1,Y_1', \cdots, Y_r'] $ obtained by giving positive weights on $U'_1,\bar u'_2,\cdots,\bar u'_d$ and weight $1$ on the $Y_i$ such that $w(Z'_i)=1$, for $1\leq i \leq r$, and $\ini_w(Z'_i)\not=Y'_i\in\gr_w k(\zeta) [U'_1,Y_1', \ldots, Y_r'] $ for at least one $i$,
and such that, for all $i$,  $1\leq i \leq r$, $\ini_w(Z'_i)=Y'_i+ \mu_i {U'_1} {\overline {u}'_2}^{a(i,2)} \cdots  {\overline {u}'_d}^{a(i,d)}$ with $\mu_i \in k$, $a(i,j) \in \Z$, $2\leq j \leq d$.
By \cite[Corollary 4.1.1 p. 286]{H}, for at least one $j$, $1\leq j \leq m$,  $\ini_w(F_j)\not=F_j(Y'_1,\ldots,Y'_r) $:
 by \eqref{eq.zeta}, the exponents $a(i,j)$ are all in $\N$. By taking all the possible $w$, in \eqref{eq.tauzeta}, for all $i, 1\leq i \leq r$,
 we get
 $\lambda_i \in\mathcal{O}_{\mathcal{Z}',x'} / (u_1', y')$.

We can find $(z_1,\ldots,z_r)\in \mathcal{O}_{\mathcal{Z}',x'}^r$ with
 $\ini_{\zeta}(z_i)=Z'_i$,  and $\Delta(f',u',z)$ has only vertices with first coordinate $>1$.
 This contradicts the fact that all vertices of  $\Delta(f',u',y')$ of abscissa $\delta-1=1$ are prepared.
 We arrived to a contradiction which proves the claim.
 \end{proof}

\noindent
{\bf \em End of the proof of Lemma~\ref{lem:lem4}.} Assume $\delta-1=1$, By Claim~\ref{cl:delta} above, the initial forms $(U'_1,Y'_1,\ldots,Y'_r)$ generate the ideal of the directrix of  $\mathcal{X}'$ at $\zeta$.
By Lemma~\ref{lem:lem3}, $(U'_1,Y'_1,\ldots,Y'_r)$ are the equations of the reduced ridge at $\zeta$: in $\mathcal{X}''$ there is no point near to $\zeta$, as the Hilbert-Samuel function is constant on $\mathcal{X}'$.
This contradicts Remark \ref{Rk:semicon}  applied to $\zeta, \mathcal{X}'$.

All  this leads to $\delta-1>1$.

\smallskip

Let $$ \pi_2 \colon \mathcal{X}'' \longrightarrow  \mathcal{X}' $$ be the blowing up of $\mathcal{X}'$ along  $ D := V(y',u_1)$. Consider the point $x'' \in \pi_2^{-1} (D) \subset {\mathcal X}'' $ of parameters
$$(u'',y''):=(u'_1,u'_2,\ldots,u'_d,y'_1/u_1',\ldots,y'_r/u_1').$$
At $ x'' $,
each vertex with smallest first coordinate of  $\Delta(f'',u'',y'')$  is not solvable with first coordinate $\delta-2>0$.
By the same argument as above, one can deduce that $\delta-2>1$. An induction on $\delta$ leads to a contradiction.
\end{proof}

\bigskip

\begin{proof}[\bf \em End of the Theorem's proof]
	Let $\mathcal{X}=\mathrm{Spec}A$ be an affine scheme, $A$ an excellent local ring and $x\in\mathcal X$ be the closed point. Assume that
the Hilbert-Samuel function is constant on $\mathcal{X}$.  By Lemma \ref{lem:lem3},
the directrix at $x$ coincides with the reduced ridge at $x$. This is  hypothesis (*) of \cite[Proposition~4.1]{CoS}.

Suppose $A=R/I$ for some excellent regular local ring $R$.
By \cite[Proposition~4.1]{CoS}, there exist a standard basis $f:=(f_1,\ldots,f_m)\in R^m$ for the ideal of $\mathcal X $ at $ x $
and
a regular system of parameters $(u,y)$ of $R$ such that  $\Delta(f;u;y)=\emptyset$.
By Lemma~\ref{lem:lem4},  $V(y)$ is permissible  for $\mathcal{X}$ at $x$.
Since the reduced ridge coincides with the directrix, the blowing up along $V(y)$ has no point near to $x$. In a neighbourhood of $x$, $V(y)$ is the Hilbert-Samuel stratum of $\mathcal{X}$ which is ${\mathcal{X}}_{red}$~: $V(y)={\mathcal{X}}_{red}$.
This ends the proof in this case.

 If $\mathcal X$ is not embedded in a regular scheme, the completion $\widehat{A}$ of the local ring $A$ at its maximal ideal is the quotient of a regular ring $R$.
 By the argument above, there exist regular parameters $(y)$ in $R$ such that $V(y)\subset \Spec (R)$
 is the Hilbert-Samuel stratum of $\widehat{\mathcal{X}}=\mathrm{Spec} \widehat{A}$. By \cite[Lemma~2.37(2)]{CoJS}, $V(y)$ is the preimage of $\mathcal{X}$'s Hilbert-Samuel stratum  which is  ${\mathcal{X}}_{red}$, since $A$ is excellent. By \cite[Lemma~2.37(2)]{CoJS}, ${\mathcal{X}}_{red}$ is regular at $x$.
\end{proof}

\begin{Rk} There exist excellent schemes  $\mathcal X$ with $\mathcal {X}_{red}$ regular and with a non-constant Hilbert-Samuel function, even if $\mathcal X$ is a complete intersection.
\end{Rk}

Look at this example: $\mathcal X\subset \Spec k[X_1,X_2,X_3]$ with ideal
$${\mathcal I}=(X_1^2+X_2X_3^2,X_2^2).$$
We have $\mathcal {X}_{red}=V(X_1,X_2)$.
On $\mathcal X$, the Hilbert-Samuel function takes different values at the origin and at the generic point.

Here is a different argument for this~: the characteristic polyhedron $\Delta({\mathcal I}; X_3)$ is not empty.
By Lemma~\ref{lem:lem4}, the Hilbert-Samuel function of $\mathcal X$ cannot be constant.


\begin{thebibliography}{99}



 \bibitem{A} {\sc Abad C.}, On the highest multiplicity locus of algebraic varieties and Rees algebras.  {\em J.~Algebra} {\bf 441} (2015), 294--313.


 \bibitem{BEV} {\sc Benito A., Encinas S., Villamayor O.}, Some natural properties of constructive resolution of singularities. {\em Asian J.~Math.} {\bf 15} (2011), no.~2, 141--191.

  \bibitem{B} {\sc Bennett B.M.}, On the Characteristic Functions of a Local Ring.  {\em Ann.~of Math.~(2)} {\bf 91} (1970), 
  25--87.

  \bibitem{BM}  {\sc Bierstone M., Milman P.},
Canonical desingularization in characteristic zero by blowing up the maximum strata
of a local invariant. {\em  Invent.~Math.} {\bf 128} (1997), no.~2, 207--302.

 \bibitem{CoJS} {\sc Cossart V., Jannsen U., Saito S.}, Desingularization: invariants and strategy -- application to dimension 2. With contributions by B.~Schober.
 Lecture Notes in Mathematics, {\bf 2270}. {\em Springer, Cham}, (2020),  viii+256 pp.

 \bibitem{CoJSc} {\sc Cossart V., Jannsen U., Schober B.}, {Invariance of Hironaka’s characteristic polyhedron}. {\em  Rev~ R.~Acad.~Cienc.~Exactas Fís.~Nat.~Ser.~A Mat. RACSAM } {\bf 113} (2019), no.~4, 4145--4169. 


\bibitem{CoP} {\sc Cossart V.,  Piltant O.}, Characteristic polyhedra of singularities without completion.
{\em Math.~Ann.} {\bf 361} (2015), no.~1-2, 157--167.

 \bibitem{CoScDim2} {\sc Cossart V.,  Schober B.},
 A strictly decreasing invariant for resolution of singularities in dimension two.
{\em Publ.~Res.~Inst.~Math.~Sci.} {\bf 56} (2020), no.~2, 217--280.


 \bibitem{CoS} {\sc Cossart V.,  Schober B.}, Characteristic polyhedra of singularities without completion: part II.
	{\em Collect.~Math.} {\bf 72} (2021), no.~2, 351--392.
	

\bibitem{Gi1} {\sc Giraud J.}, \'Etude locale des singularit\'es.
Cours de 3ème cycle, 1971–1972. Publications Mathématiques d'Orsay, No.~26. {\em Université Paris XI, U.E.R. Mathématique, Orsay}, 1972. {\rm ii}+111 pp.


\bibitem{Gi2} {\sc Giraud J.},
Contact maximal en caract\'{e}ristique positive.
{\em Ann.~Sci.~\'Ecole Norm.~Sup.~(4)} {\bf 8} (1975), no.~2, 201--234.


\bibitem{H1}
 {\sc Hironaka H.}, Resolution of Singularities of an Algebraic Variety Over a Field of Characteristic Zero: I, II.
{\em Ann.~of Math. (2)} {\bf 79} (1964), 109--203 and 205--326.

\bibitem{H2}
 {\sc Hironaka H.}, On the characters $\nu^*$ and $\tau^*$ of singularities.
{\em J.~Math.~Kyoto Univ.} {\bf 7} (1967), 19--43.


\bibitem{H}
 {\sc Hironaka H.}, Characteristic polyhedra of  singularities.
{\em J.~Math.~Kyoto Univ.} {\bf 7} (1967), 251--293.

\bibitem{P}
{\sc Pomerol M-J.}, Sur la strate de Samuel du sommet d'un cône en caractéristique positive.
{\em Bull.~Sci.~Math.~(2)} {\bf 98} (1974), no.~3, 173--182.

\bibitem{S}
{\sc Singh B.}, Effect of a permissible blowing-up on the local Hilbert functions.
{\em Invent.~Math.} {\bf 26} (1974), 201--212.

\bibitem{V}
 {\sc Villamayor O.}, Equimultiplicity, algebraic elimination, and blowing up.
{\em Adv.~Math.} {\bf 262} (2014), 313--369.

\bibitem{Z}
 {\sc Zariski O.},
Algebraic varieties over ground fields of characteristic zero.
{\em Amer.~J.~Math.}
 {\bf 62} (1940), 187--221.

\end{thebibliography}
\end{document}